\documentclass[11pt,leqno]{amsart}
\usepackage{amsthm,amsfonts,amssymb,amsmath,oldgerm,color,bm,multicol}

\usepackage{amsthm,amsfonts,amssymb,amsmath,color}

\numberwithin{equation}{section}

\renewcommand\d{\partial}

\renewcommand\o{\omega}
\newcommand\s{\sigma}

\newcommand\R{\mathbb R}\newcommand\Z{\mathbb Z}

\def\de{\delta}

\def\epsilon{\varepsilon}
\def\e{\varepsilon}

\def\nnn{\mathbf{n}}

\newcommand\br{\begin{rem}}
\newcommand\er{\end{rem}}
\newcommand\bp{\begin{pmatrix}}
\newcommand\ep{\end{pmatrix}}
\newcommand\be{\begin{equation}}
\newcommand\ee{\end{equation}}
\newcommand\ba{\begin{equation}\begin{aligned}}
\newcommand\ea{\end{aligned}\end{equation}}

\newcommand\nn{\nonumber}

\usepackage{layout}

\newcommand{\CalB}{\mathcal{B}}

\newcommand{\II}{{\mathbb I}}

\newcommand{\uu}{{\mathbf u}}

\newcommand{\ff}{{\mathbf f}}

\newcommand{\dx}{{\rm d} {x}}

\newcommand{\dive}{{\rm div\,}}

\newtheorem{defi}{Definition}[section]
\newtheorem{theorem}[defi]{Theorem}

\newtheorem{lemma}[defi]{Lemma}

\numberwithin{equation}{section}

\begin{document}

\title{Homogenization of Stokes equations in perforated domains: a unified approach}

\author{Yong Lu} 
\thanks{Department of Mathematics, Nanjing University, 22 Hankou Road, Gulou District, 210093 Nanjing, China. Email: luyong@nju.edu.cn. \\
The author warmly thanks Wenjia Jing and Christophe Prange for helpful discussions. The work of the author has been supported by the Recruitment Program of Global Experts of China. This work is partially supported by project ANR JCJC BORDS funded by l'ANR of France.}

\date{}

\begin{abstract}

We consider the homogenization of the Stokes equations in a domain perforated with a large number of small holes which are periodically distributed. In \cite{ALL-NS1,ALL-NS2}, Allaire gave a systematic study on this problem. In this paper, we introduce a unified proof for different sizes of holes for the homogenization of the Stokes equations by employing a generalized cell problem inspired by Tartar \cite{Tartar80}. 
\end{abstract}

\maketitle




\renewcommand{\refname}{References}


\section{Introduction}

Homogenization problems in the framework of fluid mechanics have gain a lot interest both in mathematical analysis and numerical analysis. Such problems represent the study of fluid flows in domains perforated with a large number of tiny holes (obstacles). The goal is to describe the asymptotic behavior of fluid flows (governed by Stokes equations, Navier-Stokes equations, etc.) as the number of holes goes to infinity and the size of holes goes to zero simultaneously. The limit equations that describe the limit behavior of fluid flows are called {\em homogenized equations} which are defined in homogeneous domains without holes.

\medskip

The perforated domain under consideration is described as follows. Let $\Omega\subset \R^d, \ d\geq 2$ be a bounded domain of class $C^{1}$.  The holes in $\Omega$ are denoted by $T_{\e,k}$ which are assumed to satisfy
\be\label{def-holes1}
B(\e x_k, \de_1 a_{\e}) \subset \subset T_{\e, k}  = \e x_k + a_{\e} T  \subset  \subset B(\e x_k, \de_2 a_{\e}) \subset\subset B(\e x_k, \de_3 \e)  \subset \subset \e Q_k,
\ee
where the cube $Q_k : = (-\frac{1}{2},\frac{1}{2})^d + k$ and $x_k = x_0 + k$ with $x_0 \in (-\frac{1}{2},\frac{1}{2})^d$, for each $k\in \Z^d$; $T$ is a model hole which is assumed to be closed, bounded, and simply connected, with $C^{1}$ boundary; $\de_{i}, \, i=1,2,3$ are fixed positive numbers.  The perforation parameters $\e$ and $a_{\e}$ are used to measure the mutual distance of holes and the size of holes, respectively,
 and $\e x_{k} = \e x_{0} + \e k$ are the locations of the holes. Without loss of generality, we assume that $x_0 = 0$ and $0< a_{\e} < \e \leq 1.$ Otherwise it is sufficient to consider the domain with a shift of $\e x_0$ and consider different values of $\de_{i}, \ i=1,2,3.$

The perforated domain (fluid domain) $\Omega_\e$ is then defined as:
\be\label{def-Oe}
\Omega_\e:=\Omega \setminus  \bigcup_{k\in K_\e} T_{\e,k}, \quad \mbox{where} \ K_\e:=\{k\in \Z^{d} \ : \ \e \overline {Q}_k\subset \Omega\}.
\ee
Throughout the paper, we consider the following Dirichlet problem of Stokes equations in $\Omega_\e$: \be\label{Stokes-Oe}
\left\{\begin{aligned}
-\Delta \uu_\e +\nabla p_\e &= \ff,\quad &&\mbox{in}~\Omega_\e,\\
\dive \uu_\e &=0 ,\quad &&\mbox{in}~\Omega_\e,\\
\uu_\e & = 0,\quad &&\mbox{on} ~\d\Omega_\e.
\end{aligned}\right.
\ee
Here we take the external force $\ff\in L^2(\Omega)$.

\medskip

For each fixed $\e>0$, the domain $\Omega_{\e}$ is bounded and is of $C^{1}$; the existence and uniqueness of the weak solution $(\uu_\e, p_{\e})\in W^{1,2}_{0}(\Omega_{\e};\R^{d}) \times L_{0}^{2}(\Omega_{\e})$ to \eqref{Stokes-Oe} is  known, see for instance \cite{Galdi-book}. Here $W^{1,2}_{0}$ denotes the Sobolev space with zero trace, and $L_{0}^{2}$ is the collection of all $L^{2}$ functions with zero average.

The behavior of the solution family $\{\uu_\e\}_{\e>0}$ as $\e\to 0$ was studied by Tartar \cite{Tartar80} for the case where the size of the holes is proportional to the mutual distance of the holes, i.e.
\be\label{case-0}
a_{\e}  =  a_{*}\e \quad \mbox{for some $a_{*} >0$}.
\ee
Then Allaire \cite{ALL-NS1, ALL-NS2} considered general cases and showed that the homogenized equations are determined by the ratio $\s_{\e}$ between the size and the mutual distance of the holes:
\be\label{ratio}
\s_{\e}: =  \left(\frac{\e^{d}}{a_{\e}^{d-2}}\right)^{\frac{1}{2}}, \quad  d\geq 3;\qquad
\s_{\e} : =  \e\left|\log \frac{a_{\e}}{\e}\right|^{\frac 12}, \quad d = 2.
\ee
Specifically, if $\lim_{\e\to 0} \s_{\e} = 0$ corresponding to the case of {\em large holes}, the homogenized system is the Darcy's law; if $\lim_{\e\to 0} \s_{\e}  = \infty$ corresponding to the case of {\em small holes}, the motion of the fluid does not change much in the homogenization process and in the limit there arise the same Stokes equations in homogeneous domains; if $\lim_{\e\to 0} \s_{\e}  = \s_{*} \in (0,+\infty)$ corresponding to the case of \emph{critical size of holes}, the homogenized equations are governed by the Brinkman's law --- a combination of the Darcy's law and the original Stokes equations.

\medskip

To obtain the limit system, a natural way is to pass $\e\to 0$ in the weak formulation of \eqref{Stokes-Oe}. In this process, one needs to pay special attention to the choice of test functions. Since the homogenized system is defined in $\Omega$,  so one needs to choose test functions in $C_c^\infty(\Omega)$. However, $C_c^\infty(\Omega)$ functions are \emph{not} proper test functions for the original system \eqref{Stokes-Oe} defined in $\Omega_{\e}$ where the test functions should be chosen in $C_c^\infty(\Omega_{\e})$. Hence, a proper surgery on the test functions need to be done and this surgery plays a crucial role in the study of the homogenization problems in fluid mechanics.  Tartar \cite{Tartar80} and Allaire \cite{ALL-NS1,ALL-NS2} used different ideas to this issue. This will be explained with more details in the next section.

\medskip

Later, the homogenization study is extended to more complicated models describing fluid flows:  Mikeli\'{c} \cite{Mik} for the incompressible Navier-Stokes equations, Masmoudi \cite{Mas-Hom} for the compressible Navier-Stokes equations,  Feireisl, Novotn\'y and Takahashi \cite{FNT-Hom} for the complete Navier-Stokes-Fourier equations. In all these studies, only the case where the size of holes is proportional to the mutual distance of the holes (like \eqref{case-0}) is considered and the Darcy's law is recovered in the limit.

Recently, cases with different sizes of holes are also considered: Feireisl, Namlyeyeva and Ne{\v c}asov{\' a} \cite{FeNaNe} studied the case with critical size of holes for the incompressible Navier-Stokes equations and they derived Brinkman's law;  in \cite{FL1, DFL, Lu-Schwarz18} the authors considered the case of small holes for the compressible Navier-Stokes equations and it is shown that the homogenized equations remain the same as the original ones. These results coincide with Allaire's study for the Stokes equations in \cite{ALL-NS1, ALL-NS2}.


\subsection{A brief review of Tartar's idea and Allaire's idea}
As pointed out in the introduction, to obtain the limit system by passing $\e\to 0$ in the weak formulation of the Stokes equations, a proper surgery on $C_{c}^{\infty}(\Omega)$  test functions needs to be done such that the test functions vanish on the holes and then become {\em good} test functions for the original Stokes equations in $\Omega_{\e}$. To this issue,  Tartar \cite{Tartar80} and Allaire \cite{ALL-NS1,ALL-NS2} used different methods.

\medskip

In \cite{Tartar80}, Tartar considered the case where the size of the holes is proportional to the mutual distance of the holes, see \eqref{case-0}. Near each single hole in $\e Q_{k}$ in the perforated domain $\Omega_{\e}$, after a scaling of size $\e^{-1}$, there arises typically the following problem, named \emph{cell problem}:
\be\label{pb-cell-Tartar}
\left\{\begin{aligned}
-\Delta w^i + \nabla q^i &=  e^i,\ &&\mbox{in}\ Q_0 \setminus  T : = \left(-\frac{1}{2},\frac{1}{2}\right)^d \setminus T,\\
\dive w^i &=0 ,\ &&\mbox{in}\ Q_0 \setminus  T,\\
w^i &=0,\ &&\mbox{on} ~  T, \\
 (w^i,q^i) & \  \mbox{is $Q_0$-periodic.}
\end{aligned}\right.
\ee
 Here $\{e^i\}_{i=1,\cdots,d}$ is the standard Euclidean coordinate of $\R^d$. The cell problem \eqref{pb-cell-Tartar} admits a unique weak solution $(w^{i},q^{i})\in W^{1,2}(Q_{0}\setminus T) \times L^{2}_{0}(Q_{0}\setminus T)$. Then by scaling back to the original scale of the perforated domain, the \emph{scaled cell solution} $(w^i_\e, q^i_\e)$ defined as
\ba\label{cell-w-i-e}
w^i_\e(\cdot) := w^i \big(\frac{\cdot}{\e}\big),\quad q^i_\e(\cdot) := q^i\big(\frac{\cdot}{\e}\big)
\ea
solves
\be\label{pb-cell-e}
\left\{\begin{aligned}
- \e^2\Delta w^i_\e + \e \nabla q^i_\e &=  e^i,\ &&\mbox{in}\ \e Q_0 \setminus  \e T ,\\
\dive w^i_\e &=0 ,\ &&\mbox{in}\ \e Q_0 \setminus  \e T,\\
w^i_\e &=0,\ &&\mbox{on} ~ \e  T,\\
 (w^i_\e, q^i_\e ) & \  \mbox{is $\e Q_0$-periodic.}
\end{aligned}\right.
\ee
Clearly $w^i_\e$ vanishes on the holes in $\Omega_\e$. Given each scalar function $\phi \in C_c^\infty(\Omega)$, $w_\e^i \phi$ is a good test function to \eqref{Stokes-Oe}. Then choosing $w_\e^i \phi$ as a test function in the weak formulation of \eqref{Stokes-Oe}, and passing $\e\to 0$, together with the property of $w_\e^i$ and the optimal uniform estimates for $\uu_\e$ and $p_\e$, gives the limit model---Darcy's law.  In this paper, we will generalize Tartar's idea so that we can cover different sizes of holes as Allaire.  So we mainly focus on the case where $a_{\e}$ is much smaller than $\e$ such that $ \eta:=\frac{a_{\e}}{\e} \to 0$ as $\e\to 0$.

\medskip

In \cite{ALL-NS1,ALL-NS2}, Allaire used an abstract framework of hypotheses on the holes and verified the hypotheses in the case of a periodic distribution of the holes. This idea goes back to \cite{Cioranescu-Murat} for Laplacian operator instead of Stokes. For general cases where $a_{\e}$ is much smaller than $\e$ such that $\eta : = \frac{a_{\e}}{\e} \to 0$ as $\e\to 0$, near each singular hole, after a scaling of size $a_{\e}^{-1}$ such that the size of the holes becomes $O(1)$, one obtains a domain of the type
$
\eta^{-1} Q_0 \setminus T
$
which converges to $\R^d \setminus T$ as $\e\to 0$.  Allaire employed the following problem of Stokes equations in exterior domain $\R^d \setminus T$, namely {\em  local problem}:
\be\label{pb-local}
\left\{\begin{aligned}
-\Delta v^i + \nabla p^i &=  0,\ &&\mbox{in}\ \R^d \setminus  T,\\
\dive v^i &=0 ,\ &&\mbox{in}\ \R^d \setminus  T,\\
v^i &=0,\ &&\mbox{on} \ T,\\
v^i &=e^i,\ &&\mbox{at infinity},
\end{aligned}\right.
\ee
to construct a family of functions $(v_\e^i, p_\e^i) \in W^{1,2}(\Omega_{\e})\times L^{2}(\Omega_{\e})$ which vanish on the holes in order to modify the $C_{c}^{\infty}(\Omega)$ test functions and derive the limit equations as $\e\to 0$.  Allaire showed that the Dirichlet problem \eqref{pb-local} is well-posed in $D^{1,2}(\R^{d}\setminus T;\R^{d}) \times L^{2}(\R^{d}\setminus T;\R^{d})$ and showed decay estimates of the solutions at infinity.  Here $D^{1,2}$ denotes the homogeneous Sobolev spaces. The corresponding $(v_\e^i, p_\e^i)$ is defined as follows: in cubes $\e Q_{k}$ that intersect with the boundary of $\Omega$,
\ba\label{def-vi-pi-1}
v_{\e}^{i} = e^{i}, \ p_{\e}^{i} = 0, \ \mbox{in $\e Q_{k} \cap \Omega$}, \quad \mbox{if} \ \e Q_{k} \cap \d \Omega \neq \emptyset;
\ea
and in cubes $\e Q_{k}$ whose closures are contained in $\Omega$,
\ba\label{def-vi-pi-2}
&v_{\e}^{i} = e^{i}, \quad  p_{\e}^{i} = 0, \ &&\mbox{in $\e Q_{k} \setminus  B(\e x_{k}, \de_{3} \e)$}, \\
& -\Delta v_{\e}^{i} +  \nabla p_{\e}^{i} = 0, \quad \dive v_{\e}^{i} = 0,\ &&\mbox{in $ B(\e x_{k}, \de_{3} \e) \setminus  B(\e x_{k}, \de_{2} \e)$}, \\
&  v_{\e}^{i} (x) =  v^{i} \big(\frac{x}{a_{\e}}\big), \quad p_{\e}^{i} (x) = \frac{1}{a_{\e}} p^{i}\big (\frac{x}{a_{\e}}\big),\ &&\mbox{in $ B(\e x_{k}, \de_{2} \e) \setminus  T_{\e,k}$}, \\
&v_{\e}^{i} = 0, \quad  p_{\e}^{i} = 0, \ &&\mbox{in $T_{\e,k}$}.
\ea
Such $(v_\e^i, p_\e^i) \in W^{1,2}(\Omega_{\e};\R^d)\times L^{2}(\Omega_{\e})$ fulfills the hypotheses in Allaire's abstract framework. In particular, $(v_\e^i, p_\e^i) $ vanishes on the holes. Thus, for each $\phi\in C_{c}^{\infty}(\Omega)$, the modified function $v_{\e}^{i} \phi$ becomes a good test function in the weak formulation of the original Stokes equations in $\Omega_{\e}$. By careful analysis, passing $\e \to 0$ gives the desired homogenized systems.

\subsection{Main result}

Tartar employed the cell problem \eqref{pb-cell-Tartar}--\eqref{pb-cell-e} to modify the test function, while he only covered the case $a_{\e}  = a_{*}  \e$ for some $a_{*}$ independent of $\e$.  Allaire employed the local problem \eqref{pb-local}--\eqref{def-vi-pi-2} and covered general sizes of holes. We found that Tartar'a idea evolving the cell problem \eqref{pb-cell-Tartar} could be more applicable when we impose soft restrictions on the distribution of the holes, which is the main topic in the forthcoming study \cite{Jin-Lu-Prange09}. Unfortunately, Tartar's method works only for a specific case. To cover the cases with general sizes of holes, a proper generalization needs to be done. Indeed, by introducing a generalized cell problem and establishing suitable estimates, we make it work for different sizes of holes. This gives a new proof of Allaire's homogenization results in \cite{ALL-NS1, ALL-NS2} by a unified approach. Along with others, such an idea of unified approach is also used recently in \cite{Jin19} for the study of Laplace equations in perforated domains. 

Before stating the theorem (see also in \cite{ALL-NS1, ALL-NS2}), we recall the extension of $(\uu_\e, p_{\e})\in W^{1,2}_{0}(\Omega_{\e};\R^{d}) \times L_{0}^{2}(\Omega_{\e})$ which is the unique solution to \eqref{Stokes-Oe} in $\Omega_{\e}$. For the velocity, since it has zero trace on the boundary, it is nature to use its zero extension defined as:
\ba\label{def-extension-u}
\tilde \uu_{\e} =\uu_{\e} \ \mbox{in}\ \Omega_\e; \quad \tilde \uu_{\e} =  0  \ \mbox{on}\ \Omega\setminus \Omega_\e, 
\ea
which satisfies
\ba\label{pt-tilde-u}
\tilde \uu_{\e} \in W^{1,2}_{0}(\Omega;\R^{d}), \quad \|\tilde \uu_{\e}\|_{L^{2}(\Omega)} = \| \uu_{\e}\|_{L^{2}(\Omega_{\e})} , \quad  \|\nabla \tilde \uu_{\e}\|_{L^{2}(\Omega)}  = \|\nabla \uu_{\e}\|_{L^{2}(\Omega_\e)}.
\ea


The extension of the pressure is more delicate and is defined by employing the so-called {\em restriction operator} due to Allaire \cite{ALL-NS1, ALL-NS2} for general sizes of holes, and due to Tartar \cite{Tartar80} for the case where the size of the holes is proportional to their mutual distance. For $\Omega_{\e}$ defined through \eqref{def-holes1} and \eqref{def-Oe}, there exists a linear operator, named {\em restriction operator}, $R_{\e} : W^{1,2}_{0}(\Omega;\R^{d}) \to W^{1,2}_{0}(\Omega_{\e};\R^{d}) $ such that:
\ba\label{pt-res}
&\uu \in W_0^{1,2}(\Omega_\e;\R^d) \Longrightarrow R_\e (\tilde \uu)=\uu \ \mbox{in}\ \Omega_\e,\ \mbox{where} \ \tilde
\uu:=\begin{cases}\uu \ &\mbox{in}\ \Omega_\e,\\ 0  \ &\mbox{on}\ \Omega\setminus \Omega_\e, \end{cases}\\
&\uu \in W_0^{1,2}(\Omega;\R^d),\  \dive \uu =0 \ \mbox{in} \ \Omega \Longrightarrow \dive R_\e (\uu) =0 \ \mbox{in} \ \Omega_\e,\\
&\uu \in W_0^{1,2}(\Omega;\R^d)\Longrightarrow \|\nabla R_\e(\uu)\|_{L^2(\Omega_\e)} \leq C \, \big( \|\nabla \uu\|_{L^2(\Omega)} + (1+\s_{\e}^{-1}) \|\uu\|_{L^2(\Omega)}\big).
\ea
Then the extension $\tilde p_\e\in L^{2}_{0}(\Omega)$ is defined through the following dual formulation:
\ba\label{def-extension-p-large}
\left<\nabla \tilde p_\e, \varphi \right>_{\Omega} = \left<\nabla  p_\e, R_\e(\varphi) \right>_{\Omega_{\e}}, \quad \forall \ \varphi \in W^{1,2}_0(\Omega;\R^d).
\ea
The above formulation \eqref{def-extension-p-large} is well defined due to the three properties in \eqref{pt-res}; moreover $\nabla \tilde p_\e \in W^{-1,2}(\Omega_\e;\R^d)$ and up to a constant, $\tilde p_\e \in L_0^2(\Omega_\e)$; in particular, $\tilde p_{\e} = p_{\e}$ in $\Omega_{\e}$.  Indeed, by property 2 of \eqref{pt-res}, one has $\dive R_{\e} (\varphi)  = 0$ for each $\varphi \in W^{1,2}_0(\Omega;\R^d)$ with $\dive \varphi = 0$, then one deduces naturally from \eqref{def-extension-p-large} that
$\left<\nabla \tilde p_\e, \varphi \right> = \left<\nabla  p_\e, R_\e(\varphi) \right> = 0.$ For each $f\in L^{2}(\Omega_{\e})$, we employ the Bogovskii operator ${\mathcal B}_{\Omega_{\e}}: L^{2}({\Omega_{\e}}) \to W^{1,2}_{0}({\Omega_{\e}};\R^{d})$ and introduce
$$
\varphi  := {\mathcal B}_{\Omega_{\e}} (f - \left< f\right>)  \in W^{1,2}_{0}(\Omega_{\e};\R^{d}) \ \mbox{with} \ \left< f\right>: = \frac{1}{|\Omega_{\e}|} \int_{\Omega_{\e}} f \,\dx
$$
such that
$
\dive \varphi  = f - \left< f\right>.
$
Let $\tilde \varphi, \tilde f$ be the zero extension of $\varphi, f$. Since $p_{\e}$ and $\tilde p_{\e}$ are both of mean zero, together with property 1 of \eqref{pt-res}, one has
\ba
\int_{\Omega_{\e}} \tilde p_{\e} f\,\dx &= \int_{\Omega} \tilde p_{\e} \tilde f\,\dx  = \int_{\Omega} \tilde p_\e  (\tilde f - \langle f \rangle )\,\dx  = \int_{\Omega} \tilde p_\e  \dive \tilde \varphi\,\dx  =  \left<\nabla \tilde p_\e, \tilde \varphi \right>_{\Omega} \\
 & = \left<\nabla  p_\e, R_\e(\tilde \varphi) \right>_{\Omega_{\e}} = \left<\nabla  p_\e,  \varphi \right>_{\Omega_{\e}}= \left< p_\e, \dive \varphi \right>_{\Omega_{\e}} = \int_{\Omega_{\e}}  p_\e  ( f - \langle f \rangle )\,\dx = \int_{\Omega_{\e}}  p_{\e} f\,\dx.
\nn
\ea
This holds for all $f\in L^{2}(\Omega_{\e})$ and therefore $\tilde p_{\e}= p_{\e}$ in $\Omega_{\e}$.

\medskip

We now state the theorem:
\begin{theorem}\label{thm} For each $\e>0$ small, let $(\uu_\e, p_{\e})\in W^{1,2}_{0}(\Omega_{\e};\R^{d}) \times L_{0}^{2}(\Omega_{\e})$ be the unique solution to the Dirichlet problem of Stokes equations \eqref{Stokes-Oe} in $\Omega_{\e}$.  Let $(\tilde \uu_{\e}, \tilde p_{\e})$ be their extension in $\Omega$ defined through  \eqref{def-extension-u}--\eqref{def-extension-p-large}.   Then we have the following description of the limit system related to different sizes of holes:
\begin{itemize}

\item[(i)] If $\lim_{\e\to 0} \s_{\e} = \infty$ corresponding to the case of small holes, then
$$
(\tilde \uu_{\e}, \tilde p_{\e}) \to (\uu,p) \ \mbox{strongly in} \ W^{1,2}_{0}(\Omega;\R^{d}) \times L_{0}^{2}(\Omega),
$$
 where $(\uu, p)$ is the unique (weak) solution to the Stokes equations:
 \be\label{Stokes-O-case1}
\left\{\begin{aligned}
-\Delta \uu +\nabla p &= \ff,\quad &&\mbox{in}~\Omega,\\
\dive \uu &=0 ,\quad &&\mbox{in}~\Omega,\\
\uu & = 0,\quad &&\mbox{on} ~\d\Omega.
\end{aligned}\right.
\ee

\item[(ii)] If $\lim_{\e\to 0} \s_{\e} = 0$ corresponding to the case of large holes, then
$$
\frac{\tilde \uu_{\e}}{\s_{\e}^{2}} \to \uu \ \mbox{weakly in} \ L^{2}(\Omega;\R^{d}), \quad  \tilde p_{\e} \to  p \ \mbox{strongly in} \  L_{0}^{2}(\Omega),
$$
 where $(\uu, p)$ satisfies the Darcy's law:
 \be\label{Darcy-O-case2}
\left\{\begin{aligned}
\uu &= A(\ff - \nabla p),\quad &&\mbox{in}~\Omega,\\
\dive \uu &=0 ,\quad &&\mbox{in}~\Omega,\\
\uu\cdot \nnn & = 0,\quad &&\mbox{on} ~\d\Omega,
\end{aligned}\right.
\ee
where $\nnn$ is the unit normal vector on the boundary of $\Omega$.

\item[(iii)] If $\lim_{\e\to 0} \s_{\e} = \s_{*} \in (0,+\infty)$ corresponding to the case of critical size of holes, then
$$
(\tilde \uu_{\e}, \tilde p_{\e}) \to (\uu,p) \ \mbox{weakly in} \ W^{1,2}_{0}(\Omega;\R^{d}) \times L_{0}^{2}(\Omega),
$$
 where $(\uu, p)$ is the unique (weak) solution to the system of Brinkman's law:
 \be\label{Brinkman-O-case2}
\left\{\begin{aligned}
-\Delta \uu +\nabla p  + \s_{*}^{-2} A^{-1} \uu &= \ff,\quad &&\mbox{in}~\Omega,\\
\dive \uu &=0 ,\quad &&\mbox{in}~\Omega,\\
\uu & = 0,\quad &&\mbox{on} ~\d\Omega.
\end{aligned}\right.
\ee

\end{itemize}

Here in \eqref{Darcy-O-case2} and \eqref{Brinkman-O-case2}, $A$ is a constant positive definite matrix given later in \eqref{A-eta-def-2}. In particular, $A$ is solely determined by the model hole $T$.
\end{theorem}


\section{Proof of Theorem \ref{thm}}
In this section, we will introduce a generalized cell problem based on the idea of Tartar \cite{Tartar80} and then give a new proof of Theorem \ref{thm} by a unified approach. Throughout the paper, we use $C$ to denote a positive constant independent of $\e$.

\subsection{Uniform estimates for $(\tilde \uu_\e, \tilde p_\e)$} \label{sec:convergence}
 We recall the estimates for $(\tilde \uu_{\e}, \tilde p_{\e})$ that have been shown in Allaire \cite{ALL-NS1, ALL-NS2}. Direct energy estimate and the properties of the restriction operator gives
 \ba\label{est-ue-pe-1}
\|\tilde \uu_\e \|_{W^{1,2}_{0}(\Omega)} \leq C, \quad \|\tilde p_\e \|_{L^{2}_{0}(\Omega)} \leq C.
\ea
Then, up to a subsequence, as $\e \to 0$:
\ba\label{lim-ue-pe-1}
\tilde \uu_\e \to \uu \ \mbox{weakly in} \ W^{1,2}_0(\Omega); \quad\tilde \uu_\e \to \uu \ \mbox{strongly in} \ L^{2}(\Omega); \quad  \tilde p_\e \to p  \ \mbox{weakly in}  \ L^2(\Omega).
\ea
The divergence free condition $\dive \uu = 0$ follows from $\dive \uu_{\e} = 0$.

\medskip

 In perforated domains, one can benefit from the zero boundary condition on the holes and obtain the following perforation version of Poincar\'e inequality (see Lemma 3.4.1 in \cite{ALL-NS2}):
 \ba\label{Poincare-Oe}
\| u\|_{L^{2}(\Omega_{\e})} \leq C \min\{1,\s_{\e}\} \|\nabla u\|_{L^{2}(\Omega_{\e})}, \quad \mbox{for each $u\in W^{1,2}_{0}(\Omega_{\e})$}.
\ea
Then for the case of large holes with $\lim_{\e\to 0} \s_{\e} = 0$, the above estimate constant in \eqref{Poincare-Oe} becomes $\s_{\e}$. By \eqref{Poincare-Oe},  direct energy estimate and the properties of the restriction operator gives
\ba\label{est-ue-pe-2-1}
\| \nabla \tilde \uu_\e \|_{L^2(\Omega)} \leq C \s_{\e},\quad \| \tilde \uu_\e \|_{L^2(\Omega)} \leq C \s_{\e}^{2},
\ea
\ba\label{est-ue-pe-2-2}
\tilde p_\e = \tilde p_\e^{(1)} +  \s_{\e} \tilde p_\e^{(2)} \ \mbox{with} \ \|\tilde p_\e^{(1)} \|_{W^{1,2}(\Omega)} + \|\tilde p_\e^{(2)} \|_{L^2(\Omega)} \leq C.
\ea
Then, up to a subsequence, as $\e \to 0$:
\ba\label{lim-ue-pe-2}
\frac{\tilde \uu_\e}{\s_{\e}^{2}} \to \uu \ \mbox{weakly in} \ L^{2}(\Omega), \quad  \tilde p_\e \to p  \ \mbox{strongly in} \ L^2(\Omega).
\ea
Since $\tilde \uu_{\e} \in W^{1,2}_{0}(\Omega)$ and $\dive\tilde \uu_\e = 0$, there holds  $\dive \uu = 0$ and $\uu \cdot \nnn = 0$ on $\d\Omega$.

\subsection{The generalized cell problem}

Near each single hole, after a scaling of size $\e^{-1}$ such that the controlling cube becomes of size $O(1)$, one obtains a domain of the form $Q_0\setminus (\eta T)$ with $\eta: = \frac{a_{\e}}{\e}$. Without loss of generality we may assume $0<\eta<1$. We then consider the following \emph{modified cell problem}:
 \be\label{pb-cell}
\left\{\begin{aligned}
-\Delta w^i_\eta + \nabla q_\eta^i &=  c_{\eta}^{2} e^i,\ &&\mbox{in}\ Q_\eta := Q_0 \setminus (\eta T),\\
\dive w_\eta^i &=0 ,\ &&\mbox{in}\ Q_\eta,\\
w_\eta^i &=0,\ &&\mbox{on} ~ \eta T,\\
 (w_\eta^i,q_\eta^i)  &\  \mbox{is $Q_0$-periodic.}
\end{aligned}\right.
\ee
Again $\{e^i\}_{i=1,\cdots,d}$ is the standard Euclidean coordinate of $\R^d$; $c_{\eta}$ is defined as
\be\label{def-c-eta}
c_{\eta} := |\log \eta|^{-\frac{1}{2}}, \quad \mbox{if $d=2$}; \quad c_{\eta} := \eta^{\frac{d-2}{2}}, \quad \mbox{if $d\geq 3$}.
\ee
Clearly $c_{\eta} \to 0$ when $\eta\to 0$. When $a_{\e}$ is proportional to $\e$, $\eta$ becomes a positive constant independent of $\e$ and $Q_{\eta}$ becomes a fixed domain of type $Q_0\setminus T$; this goes back to the case \eqref{pb-cell-Tartar} considered by Tartar. We focus on the general case $\eta  = \frac{a_{\e}}{\e} \to 0$ as $\e\to 0$. The cell problem \eqref{pb-cell} becomes singular: the domain admits a shrinking hole and becomes non-uniformly Lipschitz. This may cause the solutions to be unbounded, see \cite{Lu-Laplace, Lu-Stokes} for the cases with zero boundary conditions. 

To solve \eqref{pb-cell}, we introduce the periodic Sobolev spaces:
\be\label{periodic-spaces}
W^{1,2}_{p}(Q_0):=\{ u \in W^{1,2}(Q_0), \mbox{$u$ is $Q_0$-periodic}\},\quad W^{1,2}_{0,p}(Q_\eta):=\{ u \in W^{1,2}_p(Q_0), u=0 \ \mbox{on}\ \eta T\}. \nn
\ee
We then let $L_{0,p}^2(Q_\eta)$ be the collection of $L^2(Q_\eta)$ functions that are of zero average and $Q_0$-periodic.

For each fixed $\eta>0$, by classical theory (energy estimates and compactness), we can show there exists a unique weak solution $(w^i_\eta, q_\eta^i) \in   W^{1,2}_{0,p}(Q_\eta;\R^d) \times  L_{0,p}^2(Q_\eta)$ to \eqref{pb-cell} in the weak sense:
\ba\label{cell-weak-formulation-2}
& \int_{Q_\eta} w_\eta^i \cdot \nabla \phi \,\dx =0, \quad \forall \, \phi \in W^{1,2}_{0,p}(Q_\eta)\\
& \int_{Q_\eta} \nabla w_\eta^i :\nabla \varphi \,\dx   = c_{\eta}^{2} \int_{Q_\eta} \varphi \cdot e^i, \quad \forall \, \varphi \in W^{1,2}_{0,p}(Q_\eta;\R^d), \ \dive \varphi =0.
\ea

We shall deduce the explicit dependency of the norms $\|w_\eta^i\|_{W^{1,2}(Q_0)}$ and $\|q_\eta^i\|_{L^2(Q_0)}$ on $\eta$ when $\eta\to 0$. We focus on the case $\eta : = \frac{a_{\e}}{
\e} \to 0$ as $\e \to 0.$

\subsection{A Poincar\'e type inequality in $Q_\eta$}

We introduce the following lemma which gives a Poincar\'e type inequality in singular domain $Q_\eta$:
\begin{lemma}\label{lem-poincare-Qeta}
There exists a constant $C>0$ such that for all $u\in W^{1,2}_{0,p}(Q_\eta)$ there holds
\ba\label{poincare-Qeta}
&\|u\|_{L^2(Q_0)}  \leq   {C} c_{\eta}^{-1} \| \nabla u \|_{L^{2}(Q_{0})},
\ea
where $c_{\eta}$ is given in \eqref{def-c-eta}.
\end{lemma}

\begin{proof}
Let $u\in W^{1,2}_{0,p}(Q_\eta)$. We assume in addition $u\in C^{1}(\overline Q_{\eta})$.
 For general $u\in W^{1,2}_{0,p}(Q_\eta)$, the result follows from the classical density argument.

 By \eqref{def-holes1}, there holds
\be\label{pt-Qeta}
B(0,\de_1 \eta) \subset \eta T \subset B(0,\de_2 \eta) \subset B(0,\de_3) \subset Q_0 \subset B(0,1).
\ee

By $Q_{0}$ periodicity of $u$,  we have
\be\label{norm-Qeta-B}
 \| \nabla u \|^2_{L^2(Q_0)} \leq \| \nabla u \|_{L^2(B(0,1))}^2 \leq \| \nabla u \|_{L^2((-1,1)^d)}^2  = 2^{d} \| \nabla u \|^2_{L^2(Q_0)}.
\ee

For each $x\in B(0,1)\setminus (\eta T) \subset  B(0,1)\setminus B(0,\de_1 \eta)$, we denote $r_x := |x|$ and $\o_x : = \frac{x}{|x|}$. By the fact $u =0$ on $\eta T$,  we have
\ba\label{ux-Qeta}
u(x) & = u (r_x \o_x) = u (r_x \o_x) - u (\de_1 \eta \o_x) = \int_{\de_1\eta}^{r_x} \frac{\rm d}{{\rm d}s} u(s \o_x)\, {\rm d}s = \int_{\de_1\eta}^{r_x} (\nabla u) (s \o_x) \cdot \o_x \, {\rm d}s. \nn
\ea
By H\"older's inequality, direct calculation gives
\ba\label{ux-Qeta-L2}
\|u\|_{L^2(Q_0)}^2 &\leq \int_{B(0,1)\setminus B(0,\de_1 \eta)} |u(x)|^2 \,\dx = \int_{\de_1\eta}^{1} \int_{\mathbb{S}^2} |u(r_x \o_x)|^2 r_x^{d-1} \,{\rm d} \o_x \,{\rm d} r_x \\
&= \int_{\de_1\eta}^{1} \int_{\mathbb{S}^2} \left|\int_{\de_1\eta}^{r_x} (\nabla u) (s \o_x) \cdot \o_x \, {\rm d}s\right|^2   r_x^{d-1} \,{\rm d} \o_x \,{\rm d} r_x \\
& \leq \int_{\mathbb{S}^2} \int_{\de_1\eta}^{1} r_x^{d-1} \left(\int_{\de_1\eta}^{r_x} s^{-d+1}\,{\rm d}s\right) \left(\int_{\de_1\eta}^{r_x} s^{d-1} |\nabla u(s \o_x)|^2 \, {\rm d} s \right)  \,{\rm d} r_x \,{\rm d} \o_x \\
&\leq \left(\int_{\de_1\eta}^{1} r_x^{d-1} \left(\int_{\de_1\eta}^{r_x} s^{-d+1}\,{\rm d}s\right){\rm d} r_x\right) \left( \int_{\mathbb{S}^2}  \int_{\de_1\eta}^{1} s^{d-1} |\nabla u(s \o_x)|^2 \, {\rm d} s  \,{\rm d} \o_x \right)\\
& \leq C \int_{\de_1\eta}^{1} s^{-d+1}\,{\rm d}s \int_{B(0,1)} |\nabla u(x)|^2 \,\dx.
\ea
We then deduce from \eqref{ux-Qeta-L2} that
\ba\label{u-Qeta-L2-f1}
&\|u\|_{L^2(Q_0)}^2  \leq C |\log \eta |  \| \nabla u \|^2_{L^{2}(B(0,1))}, \quad  \mbox{if $d=2$},\\
&\|u\|_{L^2(Q_0)}^2  \leq C \eta^{-d+2}  \| \nabla u \|^2_{L^{2}(B(0,1))}, \quad \mbox{if $d\geq 3$}.
\ea
Combining \eqref{norm-Qeta-B} and \eqref{u-Qeta-L2-f1} implies our desired estimate \eqref{poincare-Qeta}.

\end{proof}

\subsection{A Bogovskii type operator in $Q_\eta$}

We then introduce a Bogovskii type operator in $Q_\eta$:
\begin{lemma}\label{lem-Bogovskii-Qeta}
There exists a linear mapping $\CalB_{Q_\eta}: L^2_{0,p}(Q_\eta) \to W^{1,2}_{0,p}(Q_\eta;\R^d)$ such that for each $f\in L^2_{0,p}(Q_\eta)$, there holds
\ba\label{Bogovskii-Qeta}
\dive \CalB_{Q_\eta}(f)=f \ \mbox{in}\ Q_\eta,\ \|\CalB_{Q_\eta}(f)\|_{W^{1,2}_{0,p}(Q_\eta)} \leq C \|f\|_{L^2(Q_\eta)}. \nn
\ea

\end{lemma}

\begin{proof}Given $f\in L^2_{0,p}(Q_\eta)$. Let $\tilde f \in L^2_{0,p}(Q_0)$ be the zero extension of $f$ in $Q_0$. Since $\tilde f$ is $Q_0$-periodic and is of zero average, we have the following expression of Fourier series:
$$
\tilde f (x) = \sum_{{\bf k}\in \Z^d\setminus \{{\bf 0}\}} f_{\bf k} e^{2\pi i {\bf k}\cdot x}, \  x\in Q_0.
$$
Here $ f_{\bf k} , \, {\bf k} \in \Z^{d}$ are the Fourier coefficients of $\tilde f$. Let
\be\label{tu-Qeta-def}\tilde u :=\nabla \Delta^{-1} f  := \sum_{{\bf k}\in \Z^d\setminus \{{\bf 0}\}} \frac{{-i \bf k}}{2\pi|{\bf k}|^2}f_{\bf k} e^{2\pi i {\bf k}\cdot x}.\nn
\ee
Then $\tilde u \in W^{1,2}_{p}(Q_{0};\R^d)$ satisfying
\be\label{tu-Qeta-pt}
\dive \tilde u = \tilde f \ \mbox{in} \ Q_0,\quad \| \tilde u \|_{W^{1,2}(Q_0)}\leq  C \|\tilde f\|_{L^2(Q_0)}.\nn
\ee

Recall \eqref{pt-Qeta} and consider the following problem in $v$ near the hole:
\be\label{pb-div-B-T-eta}
\left\{\begin{aligned}
\dive v &= \dive \tilde u = f,\ &&\mbox{in}\ B(0,\de_2 \eta)\setminus (\eta T),\\
v & =\tilde u,\ &&\mbox{on} ~ \d(B(0,\de_2 \eta)),\\
 v & =0,\ &&\mbox{on} ~ \d(\eta T).
\end{aligned}\right.
\ee
By employing the proof of Lemma 2.1.4 in Allaire \cite{ALL-NS1}, there exists a solution $v$ to \eqref{pb-div-B-T-eta} satisfying
\be\label{est-v-B-T-eta}
\|v\|_{W^{1,2}(B(0,\de_2 \eta)\setminus (\eta T))}   \leq C  \| \tilde u \|_{W^{1,2}(Q_0)}\leq  C \|\tilde f\|_{L^2(Q_0)} = C \|f\|_{L^2(Q_\eta)} . \nn
\ee

Finally, the following linear operator
\be\label{Bogovskii-Qeta-def}
\CalB_{Q_\eta} (f) : =
\left\{\begin{aligned}
& \tilde u, \  && \mbox{in} \ Q_0 \setminus B(0,\de_2 \eta),\\
& v , \  && \mbox{in} \ B(0,\de_2 \eta)\setminus (\eta T)
\end{aligned}\right.\nn
\ee
is well defined and fulfills our desired properties stated in Lemma \ref{lem-Bogovskii-Qeta}.

\end{proof}

\subsection{Estimates for $(w_\eta^i,q_\eta^i)$}

Taking $w_\eta^i$ as a test function for \eqref{pb-cell} in the weak formulation $\eqref{cell-weak-formulation-2}_{2}$ and using Lemma \ref{lem-poincare-Qeta} gives
\ba\label{est-weta-1}
& \|\nabla w_\eta^i  \|_{L^2(Q_\eta)}^2 \leq  c_{\eta}^{2} \|w_\eta^i  \|_{L^2(Q_\eta)} \leq C c_{\eta} \|\nabla w_\eta^i  \|_{L^2(Q_\eta)}.
\ea
This implies, again using Lemma \ref{lem-poincare-Qeta}, that
\ba\label{est-weta-2}
\|\nabla w_\eta^i  \|_{L^2(Q_\eta)} \leq C c_{\eta} , \quad \| w_\eta^i  \|_{L^2(Q_\eta)} \leq C.
\ea

Taking $\CalB_{Q_\eta}(q_\eta^i)$ as a test function for \eqref{pb-cell} and using Lemmas \ref{lem-poincare-Qeta} and \ref{lem-Bogovskii-Qeta} implies
\ba\label{est-qeta-1}
\|q_\eta^i  \|_{L^2(Q_\eta)}^2 & \leq   c_{\eta}^{2} \| \CalB_{Q_\eta}(q_\eta^i) \|_{L^2(Q_\eta)} + \|\nabla w_\eta^i  \|_{L^2(Q_\eta)} \|\nabla \CalB_{Q_\eta}(q_\eta^i)   \|_{L^2(Q_\eta)}.
\ea
By  Lemma \ref{lem-Bogovskii-Qeta}, \eqref{est-weta-2} and \eqref{est-qeta-1}, we get
\ba\label{est-qeta-2}
& \| q_\eta^i \|_{L^2(Q_\eta)} \leq C c_{\eta}.
\ea

\medskip

By  \eqref{est-weta-2} and compact Sobolev embedding, we have, up to a subsequence, that
\be\label{w-eta-lim}
w_\eta^i  \to  w^i  \ \mbox{weakly in} \ W^{1,2}(Q_0),\quad w_\eta^i  \to w^i  \ \mbox{strongly in} \ L^2(Q_0).
\ee

In particular, when $\eta \to 0$ as $\e \to 0$ such that $c_{\eta} \to 0$, by \eqref{est-weta-2}, there holds $\nabla w^i = 0$ meaning that the limit $w^i$ is a constant vector.

We deduce from \eqref{est-qeta-2}, up to a subsequence, that
\be\label{q-eta-lim}
c_{\eta}^{-1} q_\eta^i  \to   q^i  \ \mbox{weakly in} \ L^2(Q_{0}).
\ee

Define $A(\eta)\in M^{d\times d}$ as
\be\label{A-eta-def}
A(\eta)_{i,j} := c_{\eta}^{-2} \int_{Q_\eta} \nabla  w_\eta^i : \nabla  w_\eta^j \,\dx.\nn
\ee
Clearly $A(\eta)$ is semi-positive definite. Taking $w_\eta^j$ as a test function in \eqref{pb-cell} gives
\be\label{A-eta-def-1}
A(\eta)_{i,j} = \int_{Q_\eta}  w_\eta^j \cdot e^i \,\dx  = \int_{Q_\eta}  (w_\eta^j)_i \,\dx. \nn
\ee
By \eqref{w-eta-lim} where we have shown the weak convergence of $w^{j}_{\eta}$ in $L^{2}$ as $\eta \to 0$ up to a subsequence, we then define $A$ as the limit of $A(\eta)$:
\be\label{A-eta-def-2}
A_{i,j} = \lim_{\eta \to 0} A(\eta)_{i,j} =  \lim_{\eta \to 0} c_{\eta}^{-2} \int_{Q_\eta} \nabla  w_\eta^i : \nabla  w_\eta^j \,\dx =  \lim_{\eta \to 0} \int_{Q_\eta}  (w_\eta^j)_i \,\dx= : ( \bar w^j)_i.
\ee
We see that the matrix $A = ( \bar w^j_i)_{1\leq i,j\leq d}$ is symmetric.  Moreover, the main Theorem in \cite[Section 0]{Allaire91} says that
\be\label{A-eta-M}
\lim_{\eta \to 0} A_\eta = A =  M^{-1},
\ee
where $M$ is the permeability tensor introduced by Allaire, which is positive definite. Actually, the permeability tensor $M$ is defined by (see \cite{ALL-NS1, ALL-NS2} or \cite{Allaire91})
$$
M: = \pi \II, \ \mbox{if $d=2$};\quad M: = \Big( \frac{1}{2^d} \int_{\R^d\setminus T} \nabla v^i : \nabla v^j \,\dx \Big)_{1\leq i,j\leq d}, \ \mbox{if $d \geq 3$},
$$
where $v^{i}$ is the solution to the local problem \eqref{pb-local}. Since $M$ is uniquely determined, the convergence \eqref{A-eta-def-2} and \eqref{A-eta-M} holds for each subsequence, and then holds for the whole sequence.

{\color{blue}

}

\subsection{The scaled cell solutions}

Starting from the solution $(w^i_\eta, q_\eta^i)$ to the cell problem \eqref{pb-cell}, we define
\be\label{w-q-e-def}
w^i_{\eta,\e} (\cdot) : = w^i_{\eta} \big(\frac{\cdot}{\e}\big),\ q^i_{\eta,\e} (\cdot) : = q^i_{\eta} \big(\frac{\cdot}{\e}\big)
\ee
solving
\be\label{w-q-e-pt1}
\left\{\begin{aligned}
- \e^2 \Delta w^i_{\eta,\e} + \e \nabla q^i_{\eta,\e} & = c_{\eta}^{2} e^i,\ &&\mbox{in}\ \e Q_0 \setminus (a_\e T),\\
\dive w^i_{\eta,\e} &=0 ,\ &&\mbox{in}\ \e Q_0 \setminus (a_\e T),\\
w^i_{\eta,\e} &=0,\ &&\mbox{on}  \ a_\e T,\\
(w^i_{\eta,\e},q^i_{\eta,\e}) &\  \mbox{is $\e Q_0$-periodic}.
\end{aligned}\right.
\ee
By \eqref{est-weta-1}--\eqref{est-qeta-2}, \eqref{w-q-e-def}, direct calculation gives
\ba\label{w-q-e-pt2}
&\|w^i_{\eta,\e}\|_{L^2(\Omega)} \leq C \|w^i_{\eta}\|_{L^2(Q)} \leq C, \\
& \|q^i_{\eta,\e}\|_{L^2(\Omega)} \leq C \|q^i_{\eta}\|_{L^2(Q)} \leq C c_{\eta}, \\
&\|\nabla w^i_{\eta,\e}\|_{L^2(\Omega)} \leq C \e^{-1} \|\nabla w^i_{\eta}\|_{L^2(Q)} \leq C \e^{-1} c_{\eta} \leq C \s_{\e}^{-1},
\ea
where we observed that $ \e^{-1} c_{\eta} = \s_{\e}^{-1}$ from \eqref{ratio} and \eqref{def-c-eta}.
Thus, by the convergence we have shown in \eqref{w-eta-lim} and \eqref{q-eta-lim}, using the periodicity of $(w^i_{\eta,\e}, q^i_{\eta,\e})$, we can obtain
\be\label{w-q-e-lim}
w^i_{\eta,\e}  \to \bar w^i  \ \mbox{weakly in} \ L^{2}(\Omega),\quad c_{\eta}^{-1}q^i_{\eta,\e}  \to \bar q^i:= \int_{Q_0} q^i\,\dx  \ \mbox{weakly in} \ L^2(\Omega),
\ee
as $\e\to 0$, up to a subsequence.

\subsection{Homogenization process}

Clearly $w^i_{\eta,\e}$ vanishes on the holes in $\Omega_\e$. Given any scalar function $\phi \in C_c^\infty(\Omega)$, taking $w_{\eta,\e}^i \phi$ as a test function to \eqref{Stokes-Oe} gives
\ba\label{Stokes-Omega-periodic-weak0}
\int_{\Omega_{\e}} \nabla \uu_\e : \nabla(w_{\eta,\e}^i \phi)\,\dx - \int_{\Omega_{\e}} p_\e \, \dive(w_{\eta,\e}^i \phi)\,\dx  = \int_{\Omega_{\e}} \ff \cdot (w_{\eta,\e}^i \phi)\,\dx.
\ea
By the fact that $w_{\eta,\e}^i$ vanishes on the holes and that $(\tilde \uu_{\e}, \tilde p_{\e})$ coincides with $( \uu_{\e},  p_{\e})$ in $\Omega_{\e}$, the integral equality \eqref{Stokes-Omega-periodic-weak0}  is equivalent to
\ba\label{Stokes-Omega-periodic-weak}
\int_{\Omega} \nabla \tilde \uu_\e : \nabla(w_{\eta,\e}^i \phi)\,\dx - \int_{\Omega} \tilde p_\e \, \dive(w_{\eta,\e}^i \phi)\,\dx  = \int_{\Omega} \ff \cdot (w_{\eta,\e}^i \phi)\,\dx
\ea

We will pass $\e\to 0$ case by case in the following subsections. The limit is firstly taken up to a subsequence and we will not repeat this point.

\subsubsection{The case with small holes}

We start with the case of small holes such that $\lim_{\e \to 0} \s_{\e} \to +\infty.$ 

By \eqref{w-q-e-pt2} and \eqref{w-q-e-lim}, we have $\|\nabla w^i_{\eta,\e}\|_{L^2(\Omega)}\leq C \s_{\e}^{-1} \to 0$ as $\e\to 0$; moreover $w^i_{\eta,\e}  \to \bar w^i  \ \mbox{srtongly in} \ L^{2}(\Omega)$ by Rellich-Kondrachov compact embedding theorem. Thus, as $\e \to 0$,
\ba\label{Stokes-Omega-periodic-weak-1-1}
\int_{\Omega} \nabla \tilde \uu_\e : \nabla(w_{\eta,\e}^i \phi)\,\dx  & =  \int_{\Omega} \nabla \tilde  \uu_\e :  w_{\eta,\e}^i \otimes \nabla \phi\,\dx + \int_{\Omega} \nabla  \tilde \uu_\e : \nabla w_{\eta,\e}^i \phi\,\dx \\
& \to \int_{\Omega} \nabla \tilde  \uu :  \bar w^i \otimes \nabla \phi\,\dx = \int_{\Omega} \nabla \tilde  \uu :  \nabla(\bar w^i \phi)\,\dx ,
\ea
\ba\label{Stokes-Omega-periodic-weak-1-2}
 \int_{\Omega}\tilde  p_\e \, \dive(w_{\eta,\e}^i \phi)\,\dx   =   \int_{\Omega} \tilde p_\e \,  w_{\eta,\e}^i \cdot \nabla \phi\,\dx  \to  \int_{\Omega} p  \, \bar w^i \cdot \nabla \phi\,\dx = \int_{\Omega} p   \, \dive(\bar w^i \phi)\,\dx,
\ea
and
\ba\label{Stokes-Omega-periodic-weak-1-3}
 \int_{\Omega} \ff \cdot (w_{\eta,\e}^i \phi)\,\dx \to  \int_{\Omega} \ff \cdot \bar w^i \phi\,\dx.
\ea
Then using \eqref{Stokes-Omega-periodic-weak-1-1}--\eqref{Stokes-Omega-periodic-weak-1-3} and passing $\e\to 0$ in \eqref{Stokes-Omega-periodic-weak} implies
\ba\label{Stokes-Omega-periodic-weak-1-4}
 \int_{\Omega} \nabla \uu :  \nabla(\bar w^i \phi)\,\dx -  \int_{\Omega} p   \, \dive(\bar w^i \phi)\,\dx =  \int_{\Omega} \ff \cdot \bar w^i \phi\,\dx.\nn
\ea
This gives 
\ba\label{limit-equaton-1-weak}
\int_{\Omega} \nabla \uu : \nabla (A\varphi) - p \,\dive (A\varphi) - \ff \cdot (A\varphi) \,\dx = 0, \ \forall \, \varphi \in C_c^\infty(\Omega;\R^{d}),\nn
\ea
which means
\ba\label{limit-equaton-1}
A (-\Delta \uu + \nabla p - \ff) =0 \nn
\ea
in the weak sense. Here $A = (w^{i}_{j})_{1\leq i,j\leq d}$ is the permeability matrix defined in \eqref{A-eta-def-2} and satisfies \eqref{A-eta-M}. Since $A$ is positive definite, together with the results in \eqref{sec:convergence}, we deduce the Stokes equations in non perforated domain $\Omega$:
\ba\label{limit-equaton-1-f}
-\Delta \uu + \nabla p = \ff,  \ \dive \uu = 0 \quad \mbox{in} \ \Omega; \quad \uu = 0 \ \mbox{on} \ \d\Omega.
\ea
Since the solution $(\uu,p)\in W^{1,2}_0(\Omega;\R^d)\times L_0^2(\Omega)$ of the limit system \eqref{limit-equaton-1-f} is unique, then the limit process holds for all subsequences and then holds for the whole sequence. 

We show the strong convergence of $\tilde \uu_{\e} \to \uu$ in $W^{1,2}_{0}(\Omega;\R^{d})$. Taking $\uu_{\e}$ as a test function in the weak formulation of \eqref{Stokes-Oe}, using the property that $\tilde \uu_{\e} = \uu_{\e}$ in $\Omega_{\e}$ and the weak convergence of $\tilde \uu_{\e} \to \uu$ in $W^{1,2}_{0}(\Omega;\R^{d})$, passing $\e \to 0$ implies
\ba\label{st-conv-1}
\lim_{\e \to 0}  \| \nabla \tilde \uu_{\e}\|_{L^{2}(\Omega)}^{2}  = \int_{\Omega}  \uu \cdot \ff \,\dx.
\ea 
Taking $\uu$ as a test function to \eqref{limit-equaton-1-f} gives
\ba\label{st-conv-2}
\| \nabla \uu\|_{L^{2}(\Omega)}^{2}  = \int_{\Omega}  \uu \cdot \ff \,\dx.
\ea 
Thus $\lim_{\e \to 0}  \| \nabla \tilde \uu_{\e}\|_{L^{2}(\Omega)}  =   \| \nabla \uu\|_{L^{2}(\Omega)}$ resulting in $\nabla \tilde \uu_{\e} \to \nabla\uu$ strong in $L^{2}(\Omega)$ and finally $\tilde \uu_{\e} \to \uu$ in $W^{1,2}_{0}(\Omega;\R^{d})$.   The strong convergence $\tilde p_{\e} \to p$ in $L^{2}(\Omega)$ follows from the strong convergence $\nabla \tilde p_{\e} \to \nabla p$ in $W^{-1,2}(\Omega)$ and employing the Bogovskii operator on $\Omega$. 

\subsubsection{The case with large holes}

We then consider the case with large holes: $\lim_{\e \to 0} \s_{\e} \to 0.$  By \eqref{w-q-e-pt1}, direct calculation gives
\ba\label{Stokes-Omega-periodic-weak-2-1}
\int_{\Omega} \nabla \tilde \uu_\e : \nabla(w_{\eta,\e}^i \phi)\,\dx  & =  \int_{\Omega} \nabla \tilde  \uu_\e :  w_{\eta,\e}^i \otimes \nabla \phi\,\dx + \int_{\Omega} \nabla \tilde  \uu_\e : \nabla w_{\eta,\e}^i \phi\,\dx \\
&  =  \int_{\Omega} \nabla \tilde  \uu_\e :  w_{\eta,\e}^i \otimes \nabla \phi\,\dx + \int_{\Omega} \nabla (\phi \tilde  \uu_\e) : \nabla w_{\eta,\e}^i \,\dx - \int_{\Omega} \nabla \phi \otimes  \tilde \uu_\e : \nabla w_{\eta,\e}^i \,\dx\\
&  =  \int_{\Omega} \nabla \tilde \uu_\e :  w_{\eta,\e}^i \otimes \nabla \phi\,\dx - \int_{\Omega} \nabla \phi \otimes \tilde  \uu_\e : \nabla w_{\eta,\e}^i \,\dx \\
& \quad + \e^{-1} \int_{\Omega} \dive (\phi \tilde  \uu_\e) \, q_{\eta,\e}^i \,\dx  + \e^{-2} c_{\eta}^{2}\int_{\Omega} (\phi \tilde  \uu_\e)\cdot e^i \,\dx.
\ea
By \eqref{w-q-e-pt2}, \eqref{w-q-e-lim}, \eqref{est-ue-pe-2-1},  we have
\ba\label{Stokes-Omega-periodic-weak-2-2}
\left|\int_{\Omega} \nabla \tilde \uu_\e :  w_{\eta,\e}^i \otimes \nabla \phi\,\dx\right| & \leq C \| \nabla \tilde  \uu_\e \|_{L^2(\Omega)} \| w_{\eta,\e}^i \|_{L^2(\Omega)} \leq C \s_{\e} \to 0,\\
\left|\int_{\Omega} \nabla \phi \otimes \tilde   \uu_\e : \nabla w_{\eta,\e}^i \,\dx \right| & \leq C \|\tilde  \uu_\e \|_{L^2(\Omega)} \| \nabla w_{\eta,\e}^i \|_{L^2(\Omega)} \leq C \s_{\e}^{2} \s_{\e}^{-1}  = C \s_{\e} \to 0.\nn
\ea
Moreover, using the divergence free condition $\dive \tilde \uu_{\e} = 0$ and observing  $ \e^{-1} c_{\eta} = \s_{\e}^{-1}$ implies
\ba\label{Stokes-Omega-periodic-weak-2-4}
\left|\e^{-1} \int_{\Omega} \dive (\phi \tilde \uu_\e) \, q_{\eta,\e}^i \,\dx \right| \leq C \e^{-1}\| \tilde \uu_\e \|_{L^2(\Omega)} \| q_{\eta,\e}^i \|_{L^2(\Omega)} \leq C \e^{-1}\s_{\e}^{2} c_{\eta}  = C \s_{\e} \to 0.\nn
\ea

By \eqref{lim-ue-pe-2} and observing $\e^{-2} c_{\eta}^{2}  = \s_{\e}^{-2}$, we have
\ba\label{Stokes-Omega-periodic-weak-2-5}
 \e^{-2} c_{\eta}^{2} \int_{\Omega} (\phi \tilde \uu_\e)\cdot e^i \,\dx =  \int_{\Omega} \phi\frac{\tilde \uu_\e}{\s_\e^2} \cdot e^i \,\dx  \to \int_{\Omega} \phi \uu \cdot e^i \,\dx.\nn
\ea

For the term related to the pressure, by \eqref{est-ue-pe-2-2} and \eqref{lim-ue-pe-2}, 
\ba\label{Stokes-Omega-periodic-weak-2-6}
 \int_{\Omega}\tilde  p_\e \, \dive(w_{\eta,\e}^i \phi)\,\dx   =   \int_{\Omega} \tilde p_\e \,  w_{\eta,\e}^i \cdot \nabla \phi\,\dx  \to  \int_{\Omega} p  \, \bar w^i \cdot \nabla \phi\,\dx =  \int_{\Omega} p   \, \dive(\bar w^i \phi)\,\dx.\nn
\ea

Then passing $\e\to 0$ in \eqref{Stokes-Omega-periodic-weak} implies
\ba\label{Stokes-Omega-periodic-weak-2-f-1}
 \int_{\Omega} \phi \uu \cdot e^i \,\dx = \int_{\Omega} \ff \cdot \bar w^i \phi\,\dx + \int_{\Omega} p   \, \dive(\bar w^i \phi)\,\dx.
 \ea
Together with the results in Section \ref{sec:convergence},  from \eqref{Stokes-Omega-periodic-weak-2-f-1} we deduce the Darcy's law in $\Omega$:
 \ba\label{Stokes-Omega-periodic-weak-2-f-2}
 \uu = A(\ff - \nabla p), \ \dive \uu = 0 \quad \mbox{in} \ \Omega; \quad \uu \cdot \nnn = 0 \ \mbox{on} \ \d\Omega.
 \ea
 Since the solution $(\uu,p)\in L^2(\Omega;\R^d)\times L_0^2(\Omega)$ of the limit system \eqref{Stokes-Omega-periodic-weak-2-f-2} is uniquely determined, then the limit process holds for all subsequences and then holds for the whole sequence.

 \subsubsection{The case with critical size of holes}
 We finally consider the case $\lim_{\e\to 0} \s_{\e} = \s_{*}\in (0,+\infty)$. By \eqref{w-q-e-pt2} and \eqref{w-q-e-lim}, we have $\|w^i_{\eta,\e}\|_{W^{1,2}(\Omega)}\leq C $. Thus $w^i_{\eta,\e}  \to \bar w^i  \ \mbox{weakly in} \ W^{1,2}(\Omega)$  and $w^i_{\eta,\e}  \to \bar w^i  \ \mbox{srtongly in} \ L^{2}(\Omega)$. Together with \eqref{est-ue-pe-1}, \eqref{lim-ue-pe-1} and the strong convergence $\tilde \uu_\e \to \uu$ and $w_{\eta,\e}^i\to \bar w^{i}$ in $L^2(\Omega)$, we have for the right-hand side of \eqref{Stokes-Omega-periodic-weak-2-1}:
 \ba\label{Stokes-Omega-periodic-weak-3-1}
\int_{\Omega} \nabla \tilde \uu_\e :  w_{\eta,\e}^i \otimes \nabla \phi\,\dx & \to \int_{\Omega} \nabla \uu :  \bar w^i \otimes \nabla \phi\,\dx = \int_{\Omega} \nabla \uu :  \nabla(\bar w^i  \phi)\,\dx,\\
\int_{\Omega} \nabla \phi \otimes  \uu_\e : \nabla w_{\eta,\e}^i \,\dx & \to \int_{\Omega} \nabla \phi \otimes  \uu : \nabla \bar w^i \,\dx =0,\\
\e^{-1} \int_{\Omega} \dive (\phi \tilde \uu_\e) \, q_{\eta,\e}^i \,\dx & =\e^{-1} c_{\eta}\int_{\Omega} \dive (\phi \tilde \uu_\e) \,  (c_{\eta}^{-1} q_{\eta,\e}^i)  \,\dx  = \s_{\e}^{-1} \int_{\Omega} \nabla \phi \cdot \tilde \uu_\e \,  (c_{\eta}^{-1} q_{\eta,\e}^i) \,\dx \\
& \to  \s_{*}^{-1}\int_{\Omega} \nabla \phi \cdot \uu \, \bar q^i \,\dx = \s_{*}^{-1} \int_{\Omega} \dive ( \phi  \uu) \, \bar q^i \,\dx  = 0,\nn
\ea
where we used the fact that $\bar w^i$ and $\bar q^i$ are constant.

\medskip

Again by  the strong convergence $\tilde \uu_\e \to \uu$ and $w_{\eta,\e}^i\to \bar w^{i}$ in $L^2(\Omega)$, we obtain
 \ba\label{Stokes-Omega-periodic-weak-3-4}
 \e^{-2} c_{\eta}^{2}\int_{\Omega} (\phi \tilde \uu_\e)\cdot e^i \,\dx & =  \s_{\e}^{-2}\int_{\Omega} \phi \tilde \uu_\e \cdot e^i \,\dx  \to \s_{*}^{-2}\int_{\Omega} \phi \uu \cdot e^i \,\dx,\\
 \int_{\Omega} p_\e \, \dive(w_{\eta,\e}^i \phi)\,\dx  & =   \int_{\Omega} \tilde p_\e \,  w_{\eta,\e}^i \cdot \nabla \phi\,\dx  \to  \int_{\Omega} p  \, \bar w^i \cdot \nabla \phi\,\dx =  \int_{\Omega} p   \, \dive(\bar w^i \phi)\,\dx.\nn
\ea

Finally, passing $\e\to 0$ in \eqref{Stokes-Omega-periodic-weak} implies
\ba\label{Stokes-Omega-periodic-weak-3-f-1}
\int_{\Omega} \nabla \uu :  \nabla(\bar w^i  \phi)\,\dx + \s_{*}^{-2}\int_{\Omega} \phi \uu \cdot e^i \,\dx = \int_{\Omega} \ff \cdot \bar w^i \phi\,\dx + \int_{\Omega} p   \, \dive(\bar w^i \phi)\,\dx. \nn
 \ea
 This is the Brinkmann's law in non perforated domain $\Omega$:
 \ba\label{Stokes-Omega-periodic-weak-3-f-2}
 \s_{*}^{-2}\uu = A(\ff - \nabla p + \Delta \uu) \ \Longleftrightarrow \ - \Delta \uu + \nabla p  +  \s_{*}^{-2} A^{-1}\uu  =  \ff.
 \ea
 Moreover, by the results in Section \ref{sec:convergence}, we have
  \ba\label{Stokes-Omega-periodic-weak-3-f-3}
  \uu\in W^{1,2}_0(\Omega;\R^d), \ p\in L_0^2(\Omega), \  \dive \uu = 0.
  \ea
The solution $(\uu,p)$ to \eqref{Stokes-Omega-periodic-weak-3-f-2}--\eqref{Stokes-Omega-periodic-weak-3-f-3} is uniquely determined; therefore the limit process holds for all subsequences and then holds for the whole sequence.
 
 \medskip

 We complete the proof of Theorem \ref{thm}.


\end{document}